\documentclass[12pt]{amsart}
\usepackage[matrix,arrow]{xy}
\usepackage{amssymb}
\usepackage{amsmath}
\usepackage{amsfonts}
\usepackage{epsfig}
\usepackage{color}
\usepackage{epstopdf}
\usepackage{graphicx}
\usepackage{amsthm}
\usepackage{enumerate}
\usepackage[mathscr]{eucal}
\usepackage{verbatim}

\usepackage[bookmarks=true,hyperindex,pdftex,colorlinks,citecolor=blue,
linkcolor=blue,urlcolor=blue]{hyperref}
\usepackage{tikz}
\usetikzlibrary{matrix,arrows.meta}

\theoremstyle{plain}
\newtheorem{theorem}{Theorem}[section]
\newtheorem{cor}[theorem]{Corollary}
\newtheorem{prop}[theorem]{Proposition}
\newtheorem{lemma}[theorem]{Lemma}

\theoremstyle{definition}

\newtheorem{example}[theorem]{Example}

\newtheorem{rem}[theorem]{Remark}

\newtheorem{definition}[theorem]{Definition}

\newcommand{\R}{\mathbb{R}}
\newcommand{\N}{\mathbb{N}}

\newcommand{\eps}{\varepsilon}

\DeclareMathOperator{\diam}{diam}
\DeclareMathOperator{\co}{co}

\DeclareMathOperator{\SA}{SNA}
\DeclareMathOperator{\D}{Der}

\DeclareMathOperator{\AAP}{AAP}

\newcommand{\Lipn}{{\mathrm{Lip}}}
\newcommand{\Lip}{{\mathrm{Lip}}_0}

\title[Approximations of Lipschitz maps with maximal derivatives]
{Approximations of Lipschitz maps with maximal derivatives on Banach spaces}

\author[G.~Choi]{Geunsu Choi}
\address[G.~Choi]{Department of Mathematics Education, Sunchon National University, Jeollanam-do 57922, Republic of Korea}
\email{\texttt{gschoi@scnu.ac.kr}}


\date{\today}                                           

\begin{document}

\begin{abstract}
We study two types of approximations of Lipschitz maps with derivatives of maximal slopes on Banach spaces. First, we characterize the Radon-Nikod\'ym property in terms of strongly norm attaining Lipschitz maps and maximal derivative attaining Lipschitz maps, which complements the characterization presented in \cite{CCM}. It is shown in particular that if every Lipschitz map can be approximated by those that either strongly attain their norm or attain their maximal derivative for every renorming of the range space, then the range space must have the Radon-Nikod\'ym property. Next, we prove that every Lipschitz functional defined on the real line can be locally approximated by maximal affine functions, while such an approximation cannot be guaranteed in the context of uniform approximation. This extends the previous work in \cite{BJLPS} in view of maximal affine functions.
\end{abstract}

\maketitle

\section{Introduction}

One of the most fundamental areas of functional analysis is the study of nonlinear geometric aspects of Banach spaces. In particular, the investigation of Lipschitz function spaces is not only historically deep, but also remains highly active today. It is undeniable that the study of differentiability of Lipschitz functions was a prominent topic in the past decades as evidenced by \cite{Ar,BJLPS,M,P} for instance. On the one hand, the problem of whether bounded linear operators attain their norm or not, a related area of interest, has recently shifted its focus to the case of Lipschitz maps. We refer to \cite{CCGMR,CGMR,CM} for recent studies on approximations of norm attaining Lipschitz maps, and further to \cite{C,G2,KMS} for studies on various Lipschitz norm attainments. In \cite{G2}, Godefroy introduced a concept which merges the notions of norm attainment and differentiability of Lipschitz maps, which we will call maximal derivative attaining Lipschitz maps, and this will serve as the foundation for our primary interest. With the help of this definition, we mainly aim to give new consequences on the approximation of Lipschitz maps with maximal derivatives which improve the results given in \cite{CCM}. Another concept of approximation by differentiable functions which we deal in this paper is the property so-called approximation by affine property introduced in \cite{BJLPS}, where the complete characterization of Lipschitz function spaces for which the property holds was established (see also \cite[Definition 5.19]{BL}). Considering the perspective of maximal derivatives, we will extend their results to approximation by affine functions which has almost the same Lipschitz number as the original Lipschitz map.

Throughout the article, the letters $X$ and $Y$ will always denote real Banach spaces, excluding trivial Banach spaces. In order to ensure clarity, let us briefly review the basic notations which will appear repeatedly. We denote by $S_X$ and $B_X$ the unit sphere and unit ball of $X$, respectively. By the \emph{Lipschitz map} $f \in \Lipn(X,Y)$ we mean a function from $X$ into $Y$ which satisfies that the \emph{Lipschitz number} defined by
$$
\Lipn(f) := \sup_{(p,q) \in \widetilde{X}} \frac{\|f(p)-f(q)\|}{\|p-q\|}
$$
is finite, where $\widetilde{X}$ stands for the set $\widetilde{X} = \{(p,q) \in X^2: p \neq q\}$. Since the Lipschitz number is a seminorm rather than a norm in the space of all Lipschitz maps in general, we sometimes restrict our focus to Lipschitz maps which vanishes at $0 \in M$. In this case, we write the space of Lipschitz maps from $X$ into $Y$ as $\Lip(X,Y)$ equipped with the \emph{Lipschitz norm} $\|\cdot\|$ to avoid the ambiguity. Recently, a particular type of Lipschitz map which has a maximal derivative was introduced in \cite{G2} with a different terminology.

\begin{definition}\cite{G2}\label{definition:Der}
We say that a Lipschitz map $f \in \Lip(X,Y)$ \emph{attains its maximal derivative} if there exist $x \in X$ and  $e \in S_X$ such that
$$
f'(x,e) := \lim_{t \to 0} \frac{f(x+te)-f(x)}{t} \in Y \text{ exists and } \|f'(x,e)\|=\|f\|.
$$
In this case, we write $f \in \D(X,Y)$. Sometimes one might consider the one-sided limit instead of the aforementioned definition, denoted as $\text{Der}^+(X,Y)$. Regardless, in all subsequent results, we can substitute $\text{Der}^+(X,Y)$ for the original definition.
\end{definition}

In fact, the notation $f'(x,e)$ was considered in \cite{P} to represent the directional derivative of a Lipschitz function at the point $x \in X$ in the direction of $e \in S_X$. In the past decade, the norm attainment of Lipschitz maps has been extensively studied as previously referenced to gain a deeper understanding the geometric properties of Lipschitz function spaces. A Lipschitz map $f \in \Lip(X,Y)$ is said to \emph{strongly attain its norm} if there exists $(p,q) \in \widetilde{X}$ such that
$$
\|f\| = \frac{\|f(p)-f(q)\|}{\|p-q\|},
$$
and denoted by $f \in \SA(X,Y)$. A wealth of literature on the strong norm attainment of Lipschitz maps can be found in a recent series of papers as it is the most intuitive definition, at least it seems, in light of the strong parallels between the norm attainment for bounded linear operators. In \cite{CCM}, the relationships among various types of norm attaining Lipschitz maps were classified, including the connection between strongly norm attaining Lipschitz maps and maximal derivative attaining Lipschitz maps. Recall that a Banach space $X$ is said to have the \emph{Radon-Nikod\'ym property} \textup{(RNP} in short\textup{)} \cite{DU} if for every finite measure space $(\Omega,\Sigma,\mu)$ and for every $\mu$-continuous measure $G: \Sigma \to X$ of bounded variation, there exists $g \in L_1(\mu,X)$ such that $G(E) = \int_E g \,d\mu$ for all $E \in \Sigma$. It is worth noting that the \textup{RNP} plays a significant role in characterizing the geometric structure of Banach spaces in terms of norm attaining operators (see \cite{B,CCJM} for instance). Based on Definition \ref{definition:Der}, the following sequence of results presented in \cite{CCM} provides various characterizations, which will be of fundamental importance.

\begin{theorem}\cite{CCM}\label{theorem:CCM}
Let $X$ and $Y$ be Banach spaces.
\begin{enumerate}
\item[\textup{(a)}] If $Y$ has the \textup{RNP}, then $\SA(X,Y) \subseteq \D(X,Y)$.
\item[\textup{(b)}] $Y$ has the \textup{RNP} if and only if $\D(\R,Y)$ is dense in $\Lip(\R,Y)$.
\item[\textup{(c)}] If $\D(X,Y)$ is dense in $\Lip(X,Y)$ for some Banach space $X$, then $Y$ has the \textup{RNP}.
\end{enumerate}
\end{theorem}

Our main objective in Section \ref{section:Der} is to investigate the converse and generalizations of Theorem \ref{theorem:CCM}. The study in this direction will commence with Theorem \ref{theorem:RNP-Der}, which establishes that the converse of Theorem \ref{theorem:CCM} (a) holds only if $Y$ has the extremal \textup{RNP} (see Definition \ref{definition:extremal-RNP}). Then by observing that the sets $\SA(X,Y)$ and $\D(X,Y)$ are different in many cases, we extend Theorems \ref{theorem:CCM}.(b) and \ref{theorem:CCM}.(c) in terms of the union of those two sets.

On the one hand in \cite{BL}, the existence of certain Lipschitz function spaces where every Lipschitz map can be approximated by affine functions defined on a smaller domain is explored. Below we present the precise definition of the concept we will extend. Recall that the diameter of a bounded set $D$ is denoted by $\diam(D)$.

\begin{definition}\cite{BL}
A pair $(X,Y)$ of Banach spaces is said to have the \emph{approximation by affine property} \textup{(}$\AAP$ in short\textup{)} if for every $\eps>0$, a Lipschitz map $f \in \Lipn(X,Y)$ and a ball $B \subseteq X$, there exists another ball $B_1 \subseteq B$ and an affine function $g: B_1 \to Y$ such that
$$
\sup_{x \in B_1} \|g(x)-f(x)\| \leq \eps \diam(B_1) \Lipn(f).
$$
If we can find a function $c(\eps)>0$ such that $B_1$ can be chosen $\diam(B_1) \geq c\diam(B)$, then we say $(X,Y)$ has the \emph{uniform approximation by affine property} \textup{(}\emph{uniform} $\AAP$ in short\textup{)}.
\end{definition}

The uniform $\AAP$ was fully characterized in \cite{BJLPS} by showing that the pair $(X,Y)$ has the uniform $\AAP$ if and only if either $X$ or $Y$ is super-reflexive and the other one is finite-dimensional. However, in our context, since the \textup{AAP} does not ensure that the resulting function has a slope sufficiently close to that of the approximating function, a property with a stronger assumption in terms of maximal derivatives will be introduced in Section \ref{section:maximal-AAP}. As main results, we will see that the maximal $\AAP$ is obtained for the pair $(\R,\R)$, while the same pair $(\R,\R)$ fails to satisfy the maximal uniform $\AAP$, which highlights a significant difference from the uniform $\AAP$.

Let us finish the section with introducing the notion of a fat Cantor set, as many of our counterexamples are constructed based on this notion. Recall that the \emph{Cantor set} is the subset of $[0,1]$ constructed by iteratively removing the middle third of each remaining interval. More precisely, we construct a decreasing sequence of subsets $(C_n)$ of $[0,1]$ as follows:
$$
C_1 := [0,1/3] \cup [2/3,1] \subseteq [0,1],
$$
$$
C_2 := [0,1/9] \cup [2/9,1/3] \cup [2/3,7/9] \cup [8/9,1] \subseteq [0,1],
$$
$$
\vdots
$$
and define the Cantor set $C$ by $C = \bigcap_{n=1}^\infty C_n$. It is straightforward to verify that $C$ is a nonempty nowhere dense compact subset of $[0,1]$ with $\lambda(C)=0$, where $\lambda$ denotes the Lebesgue measure on $\R$. A \emph{fat Cantor set} is defined analogously; we remove the middle third of each interval on $[0,1]$, but with a decreasing proportion of cutouts at each step so that the obtained set satisfies $\lambda(C)>0$. Still, the resulting $C$ remains a nonempty nowhere dense compact subset of $[0,1]$, which will be essential for our constructions.

\ \

\section{On the maximal derivative attaining Lipschitz maps}\label{section:Der}

We begin by introducing the main definition, which serve as a necessary condition for our results as noted in the introduction. Explicit constructions and examples will be provided subsequently. Recall that a bounded subset $D \subseteq X$ is said to be \emph{dentable} if for every $\eps>0$, there exists $x \in D$ such that $x \notin \overline{\co}(D \setminus B(x,\eps))$ where $B(x,\eps)$ is the open ball centered at $x$ with a radius $\eps$. It is well known that $X$ has the \textup{RNP} if and only if every bounded closed convex subset of $X$ is dentable (see \cite[p.218]{DU} for instance). Recall that for a bounded convex set $D \subseteq X$, a point $z \in D$ is called a \emph{farthest point} if there exists $x \in X$ such that
$$
\|z-x\| = \sup_{y \in D} \|y-x\|.
$$
The study of farthest points of convex sets has been extensively investigated, as exemplified by \cite{As,BD,L,Z}.

\begin{definition}\label{definition:extremal-RNP}
A Banach space $X$ is said to have the \emph{extremal} \textup{RNP} if every bounded closed convex subset of $X$ which contains a farthest point is dentable.
\end{definition}

\begin{rem}
It is shown in \cite{As} that if $X$ is reflexive and locally uniformly convex, then every bounded closed convex subset $D$ of $X$ contains a farthest point. On the contrary, it is well known that there exists a Banach space which contains a bounded closed convex subset with no farthest points (see \cite{BD} for instance).
\end{rem}

It is clear that the \textup{RNP} implies the extremal \textup{RNP}, but the converse is not known. Still, we will see that many classical Banach spaces without the \textup{RNP} fail to have the extremal \textup{RNP} by considering their unit balls.

\begin{example}\label{example:extremal-RNP}
$c_0$, $L_1[0,1]$, $L_\infty[0,1]$ and $C[0,1]$ fail the extremal \textup{RNP}.
\end{example}

\begin{proof}
For $c_0$, $L_1[0,1]$ and $C[0,1]$, see \cite{R}. For $L_\infty[0,1]$, see \cite[Example V.3.5]{DU}.
\end{proof}

On the one side, the next result shows in particular that $\ell_\infty$ also fails the extremal \textup{RNP}.

\begin{prop}
Let $X$ be a separable Banach space such that $X^*$ is non-separable. Then, $X^*$ fails the extremal \textup{RNP}.
\end{prop}

\begin{proof}
By \cite[Theorem 5.23]{BL}, $B_{X^*}$ contains an $\eps$-bush for some $\eps>0$. An inspection of the proof shows that we can choose $G_0 := B_{X^*}$ and $x_0^* \in G_0$ with $\|x_0^*\|=1$.
\end{proof}

This gives an affirmative answer to the above question for duals of separable spaces.

\begin{cor}\label{cor:sep-dual}
Let $X$ be a separable Banach space. Then, $X^*$ has the \textup{RNP} if and only if $X^*$ has the extremal \textup{RNP}.
\end{cor}

For the next remark, we note that it is not known whether the extremal \textup{RNP} is an isomorphic property or not. Nevertheless, for a Banach space which fails the \textup{RNP}, we can renorm the space to fail the extremal \textup{RNP}.

\begin{prop}\label{prop:renorm}
Let $X$ be a Banach space which fails the \textup{RNP}. Then, $X$ can be renormed to fail the extremal \textup{RNP}.
\end{prop}

\begin{proof}
Let $D$ be a bounded closed convex subset of $X$ which is not dentable. Consider a Banach space isomorphic to $X$ whose unit ball is not dentable (see \cite[Proposition 1]{DP}), and the rest follows as unit balls always contain farthest points.
\end{proof}

In order to discuss the results on characterizations of maximal derivative attaining Lipschitz maps, we introduce the following essential lemma of constructing a very special function when the range space fails the extremal \textup{RNP}. Although the main techinques are briefly outlined in the references cited in the proof, we provide more detailed explanations in our argument, as they will serve as the cornerstone for constructing counterexamples.

\begin{lemma}\label{lemma:RNP-martingale}
Let $Y$ be a Banach space which fails the extremal \textup{RNP}. Then, there exists $\eps>0$ and a finite $B_Y$-valued martingale $(f_n,B_n)$ in $L_1([0,1],Y)$ where $B_n$ is a subfield consisting of intervals in $[0,1]$ such that $f_1$ is a constant function in $S_{L_1([0,1],Y)}$ and $\|f_n(t)-f_{n+1}(t)\| \geq \eps$ for each $n \in \N$ and $t \in [0,1]$. Moreover, the set of endpoints of intervals in $\bigcup B_n$ forms a dense subset in $[0,1]$.
\end{lemma}

\begin{proof}
Under suitable contraction and translation, let $D$ be a bounded closed convex non-dentable subset of $B_Y$ which contains a farthest point $z \in S_Y$ so that
$$
\|z\| = \sup_{y \in D} \|y\|.
$$
For the construction of $\eps>0$ and $f_n$, see \cite[Theorem 5.21]{BL} and the proof of \cite[Theorem 5.8]{BL}. We have that all the values of $f_n$ lie in $D$, and so it is only relevant to choose $f_1$ in $S_{L_1([0,1],Y)}$, which can be achieved by defining $f_1 \equiv z$. Since each $f_n$ is $B_Y$-valued, the set of endpoints of intervals in $\bigcup B_n$ is dense in $[0,1]$ (see the proof of \cite[Theorem 5.21]{BL} and also Chapter V.3 in \cite{DU} for instance).
\end{proof}

Following the approach outlined in Lemma \ref{lemma:RNP-martingale}, we are able to establish the reverse inclusion in terms of distinguished sets of norm attaining Lipschitz maps, extending \cite[Proposition 2.1]{CCM}.

\begin{theorem}\label{theorem:RNP-Der}
Let $X$ and $Y$ be Banach spaces. If $\SA(X,Y) \subseteq \D(X,Y)$, then $Y$ has the extremal \textup{RNP}.
\end{theorem}

\begin{proof}
We will first show that $\SA(\R,Y) \not\subseteq \D(\R,Y)$ if $Y$ fails the extremal \textup{RNP}. Since $Y$ fails the extremal \textup{RNP}, according to Lemma \ref{lemma:RNP-martingale} we may consider a $B_Y$-valued martingale $(f_n,B_n)$ in $L_1([0,1],Y)$ such that $\|f_n(t)-f_{n+1}(t)\| \geq \eps$ for each $n \in \N$ and $t \in [0,1]$. Let us say $f_1 \equiv z$ where $z$ is as in Lemma \ref{lemma:RNP-martingale}. Since $(f_n,B_n)$ is a martingale in $L_1([0,1],Y)$ and the set of endpoints of intervals in $B_n$ is dense in $[0,1]$, we may define the function $g: [0,1] \to Y$ continuously by
\begin{equation}\label{equation:define-g}
g(t) := \lim_n \int_0^t f_n(\tau) \,d\tau.
\end{equation}
Here, we think of $g: [0,1] \to Y$ as a function in $\Lip(\R,Y)$, with its values extended constantly to the endpoints of $[0,1]$. It is clear that $g \in \SA(\R,Y)$ with $\|g\|=1$ since $g(1)=z$ (in fact by Lemma \ref{lemma:SNA-line}, this shows that $g$ strongly attains its norm at $(t_1,t_2)$ with any $t_1<t_2$ in $[0,1]$). Suppose now that $g \in \D(\R,Y)$. Then, there exists $t_0 \in (0,1)$ such that
$$
g'(t_0) = \lim_{h \to 0} \frac{g(t_0+h)-g(t_0)}{h} \text{ exists and } \|g'(t_0)\|=1.
$$
Fix $\delta>0$ so that
\begin{equation}\label{equation:if-Der}
\left\| \frac{g(t_0+h)-g(t_0)}{h} -g'(t_0) \right\| < \frac{\eps}{2}
\end{equation}
whenever $0 < |h|<\delta$.

\underline{Case 1}: For some large $n_0 \in \N$, suppose there exist atomic intervals $[s_1,s_2] \in B_{n_0+1}$ and $[t_1,t_2] \in B_{n_0}$ with sufficiently small length such that $t_0 \in (s_1,s_2) \subseteq (t_1,t_2) \subseteq (t_0-\delta,t_0+\delta)$. If this is the case, by the martingale property of $(f_n,B_n)$ we have
$$
g(s_2)-g(s_1) = \int_{s_1}^{s_2} f_{n_0+1} (t_0) \,d\tau \quad \text{and} \quad g(t_2)-g(t_1) = \int_{t_1}^{t_2} f_{n_0} (t_0) \,d\tau.
$$
Therefore, using \eqref{equation:if-Der} and from
$$
\dfrac{g(t_2)-g(t_1)}{t_2-t_1} = \dfrac{t_2-t_0}{t_2-t_1} \cdot \dfrac{g(t_2)-g(t_0)}{t_2-t_0} + \dfrac{t_0-t_1}{t_2-t_1} \cdot \dfrac{g(t_0)-g(t_1)}{t_0-t_1},
$$
we have actually obtained that
$$
\eps \leq \|f_{n_0}(t_0)-f_{n_0+1}(t_0)\| = \left\| \frac{g(t_2)-g(t_1)}{t_2-t_1} - \frac{g(s_2)-g(s_1)}{s_2-s_1} \right\| < \eps,
$$
which is impossible.

\underline{Case 2}: Otherwise, suppose there exists a nonincreasing sequence $(t_n) \subseteq (t_0,1)$ converging to $t_0$ such that each interval $[t_0,t_n]$ is atomic in $B_n$ for each sufficiently large $n \in \N$ without loss of generality. So for a sufficiently large $n_0 \in \N$ and $t_0<s<t_{n_0+1}$, we have a contradiction
\begin{align*}
\eps &\leq \left\|f_{n_0}(s) - f_{n_0+1}(s)\right\| \\
&= \left\| \frac{g(t_{n_0})-g(t_0)}{t_{n_0}-t_0} - \frac{g(t_{n_0+1})-g(t_0)}{t_{n_0+1}-t_0} \right\| < \eps.
\end{align*}

In order to obtain the promised result, fix any $z_0^* \in S_{X^*}$ which attains its norm at $z \in S_X$ and define the function $\tilde{g} \in \Lip(X,Y)$ by
$$
\tilde{g}(x) := g(z_0^*(x))
$$
for $x \in X$, where $g \in \Lip(\R,Y)$ is defined in (\ref{equation:define-g}). It is immediate that $\|\tilde{g}\|=1$ with $\|\tilde{g}(z)-\tilde{g}(0)\|=1$. Assume similarly $\tilde{g} \in \D(X,Y)$, that is, there exists $x_0 \in X$ and $e_0 \in S_X$ such that $\tilde{g}'(x_0,e_0)$ exists and $\|\tilde{g}'(x_0,e_0)\|=1$. Fix $\delta>0$ such that
$$
\left\| \frac{\tilde{g}(x_0+he_0)-\tilde{g}(x_0)}{h} -\tilde{g}'(x_0,e_0) \right\| < \frac{\eps}{2}
$$
whenever $0<|h|<\delta$. We may assume that $|z_0^*(e_0)| = 1$, otherwise
$$
\left\| \frac{g(z_0^*(x_0+he_0))-g(z_0^*(x_0))}{h} \right\| = |z_0^*(e_0)| \left\| \frac{g(z_0^*(x_0+he_0))-g(z_0^*(x_0))}{z_0^*(x_0^*+he_0) - z_0^*(x_0)} \right\| \leq |z_0^*(e_0)| \|g\|
$$
does not converge to 1. If $z_0^*(e_0) = - 1$, just consider $-e_0$ instead of $e_0$. It follows that we can again choose suitable atomic intervals in $B_{n_0}$ whose lengths are sufficiently small. Analogously, there exist $(s_1,s_2) \subseteq (t_1,t_2)$ and $s \in (z_0^*(x_0) + s_1,z_0^*(x_0) + s_2)$ such that $[z_0^*(x_0) + s_1,z_0^*(x_0) + s_2] \in B_{n_0+1}$, $[z_0^*(x_0) + t_1,z_0^*(x_0) + t_2] \in B_{n_0}$ and $(z_0^*(x_0) + t_1,z_0^*(x_0) + t_2) \in (s-\delta,s+\delta)$, which implies that
\begin{align*}
\eps &\leq \|f_{n_0}(s) - f_{n_0+1}(s)\| \\
&= \left\| \frac{g(z_0^*(x_0+t_2e_0))-g(z_0^*(x_0+t_1e_0))}{t_2-t_1} - \frac{g(z_0^*(x_0+s_2e_0))-g(z_0^*(x_0+s_1e_0))}{s_2-s_1} \right\| < \eps.
\end{align*}
\end{proof}

Consequently, Theorem \ref{theorem:RNP-Der} in conjuction with Proposition \ref{prop:renorm} yields the following characterization on the \textup{RNP}.

\begin{cor}
Let $X$ and $Y$ be Banach spaces. Then, the following statements are equivalent.
\begin{enumerate}
\item[\textup{(a)}] $Y$ has the \textup{RNP}.
\item[\textup{(b)}] $\SA(X,Z) \subseteq \D(X,Z)$ holds for every renorming $Z$ of $Y$.
\end{enumerate}
\end{cor}

For duals of separable spaces, the characterization can be expressed as follows with the aid of Corollary \ref{cor:sep-dual}.

\begin{cor}
Let $X$ and $Y$ be Banach spaces and $Y$ is separable. Then, the following statements are equivalent.
\begin{enumerate}
\item[\textup{(a)}] $Y^*$ has the \textup{RNP}.
\item[\textup{(b)}] $\SA(X,Y^*) \subseteq \D(X,Y^*)$.
\end{enumerate}
\end{cor}

Recall that a Lipschitz map $f \in \Lip(X,Y)$ \emph{attains its norm locally directionally} \cite{KMS} if there exist a sequence of pairs $((p_n,q_n)) \subseteq \widetilde{X}$ and points $y \in \|f\| S_Y$, $u \in S_X$ and  $v \in X$ such that
$$
\lim_n \frac{f(p_n)-f(q_n)}{\|p_n-q_n\|} = y, \quad \lim_n \frac{p_n-q_n}{\|p_n-q_n\|} = u \quad \text{and} \quad \lim_n \,(p_n, q_n) = (v, v).
$$
As an easy observation, every maximal derivative attaining Lipschitz map attains its norm locally directionally. The study of relation between the sets $\SA(X,Y)$ and $\D(X,Y)$ was first considered in \cite{CCM}. They presented an explicit example of a Lipschitz map $f \in \Lip(\R,L_1(\R))$ such that $f$ strongly attains its norm but fails to attain its maximal derivative, and in fact, does not even attain its norm locally directionally (see \cite[Example 2.5]{CCM}). However, we will discover in the next example that this is not always the case. Observe from Example \ref{example:extremal-RNP} and Theorem \ref{theorem:RNP-Der} that $\SA(\R,c_0) \not\subseteq \D(\R,c_0)$.

\begin{example}\label{example}
Consider the Lipschitz map $f \in \Lip(\R,c_0)$ defined by
$$
f(t) := \max\{\min\{t,1\},0\}\, e_1 + \sum_{n=2}^\infty \sum_{j=1}^{2^{n-2}} \max \left\{ \frac{1}{2^{n-1}} - \left| t - \frac{2j-1}{2^{n-1}} \right| ,0 \right\}e_n \quad \text{for } t \in \R.
$$

\begin{figure}[h]
\centering
\includegraphics[width=0.8\textwidth]{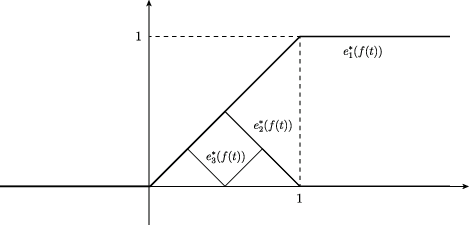}
       \caption{Visual representation of Example \ref{example}}
\end{figure}
We will directly show that $f$ strongly attains its norm while it does not attain its maximal derivative. To see this, notice first that $f$ is well-defined since for any $t \in \R$ and $n \in \N$, $|e_n^*(f(t))| \leq 1/2^{n-1}$ where $(e_n^*)$ is a standard basis of $\ell_1$. Given any $t_1 < t_2$ in $[0,1]$, it is clear that
$$
\|f(t_2)-f(t_1)\| \geq \|(t_2-t_1)e_1\| = t_2-t_1,
$$
so $\|f\| \geq 1$. On the other hand, for any fixed $n \in \N$, we have $|e_n^*(f(t_2)-f(t_1))| \leq |t_2-t_1|$ which implies $\|f\|=1$. It remains to verify that $f$ does not attain its maximal derivate at any point $t \in [0,1]$.

Suppose that there exists $t_0 \in [0,1]$ such that $f'(t_0)$ exists and $\|f'(t_0)\|=1$. As $f'(t_0) \in S_{c_0}$, we may find $n_0 \in \N$ so that $|e_{n_0}^*(f'(t_0))| < 1/2$. Now, choose $\delta>0$ such that whenver $0<|h|<\delta$, we have
$$
\left\| f'(t_0) - \frac{f(t_0+h)-f(t_0)}{h} \right\| < \frac{1}{2}.
$$
From the construction of $f$, assuming $t_0 < 1$ it is possible to find $0< h_0 < \delta$ such that
$$
\left|e_{n_0}^*\left(\frac{f(t_0+h_0)-f(t_0)}{h_0}\right)\right| = 1.
$$
However, this contradicts the fact that
$$
\left| e_{n_0}^* \left( f'(t_0) - \frac{f(t_0+h_0)-f(t_0)}{h_0} \right) \right| < \frac{1}{2}.
$$

On the other hand, define the sequence of distinct points $((p_n,q_n)) \subseteq \widetilde{\R}$ by
$$
(p_n,q_n) = \left( \frac{1}{2} + \frac{1}{2^n}, \frac{1}{2} - \frac{1}{2^n} \right) \quad \text{for each } n \in \N.
$$
By the construction of $f$, it is not difficult to see $e_k^*(f(p_n))=e_k^*(f(q_n))$ for every $n \in \N$ and $k \geq 2$. This gives us that
\begin{align*}
\lim_n \frac{f(p_n)-f(q_n)}{|p_n-q_n|} &= \lim_n \sum_{k=1}^\infty e_k^* \left( \frac{f(p_n)-f(q_n)}{|p_n-q_n|} \right) e_k \\
&= \lim_n e_1^* \left( \frac{f(p_n)-f(q_n)}{|p_n-q_n|} \right) e_1 = e_1.
\end{align*}
Since it is clear that $p_n-q_n = |p_n-q_n|$ for every $n \in \N$ and $\lim_n \,(p_n,q_n) = (1/2,1/2)$, this shows that $f$ attains its norm locally directionally. In fact, the Lipschitz map $f$ described above can be obtained by following the arguments in the proof of Theorem \ref{theorem:RNP-Der}, although we present it here in a more explicit manner.
\end{example}

We now present the main result regarding denseness, based on the direction of Theorem \ref{theorem:CCM}, which asserts that having the extremal \textup{RNP} is a necessary condition for the range space to ensure that every Lipschitz map can be approximated by those that either strongly attain their norm or attain their maximal derivative. It is noteworthy that this result derives its significance from the fact that the sets $\SA(X,Y)$ and $\D(X,Y)$ are both proper subsets of $\SA(X,Y) \cup \D(X,Y)$ precisely when $Y$ lacks the extremal \textup{RNP}, as observed by Theorem \ref{theorem:RNP-Der}. Moreover, all previously known examples of Lipschitz maps in $\Lip(X,Y)$ which cannot be approximated by strongly norm attaining (resp. maximal derivative attaining) Lipschitz maps belong to $\D(X,Y)$ (resp. $\SA(X,Y)$). We first introduce the following lemma, which is a trivial extension of \cite[Lemma 2.2]{KMS}.

\begin{lemma}\label{lemma:SNA-line}
Let $X$ and $Y$ be Banach spaces, and let $f \in \SA(X,Y)$ strongly attain its norm at some $(p,q) \in \widetilde{X}$. Then for any $r=(1-\lambda) p + \lambda q$ with $0< \lambda <1$, $f$ strongly attains its norm at both $(p,r)$ and $(r,q)$.
\end{lemma}

\begin{proof}
Observe that
\begin{align*}
\|f\| \|p-q\| = \|f(p)-f(q)\| &\leq \|f(p)-f(r)\| + \|f(r)-f(q)\| \\
&\leq \|f\| ( \|p-r\| + \|r-q\|) \\
&= \|f\| \|p-q\|.
\end{align*}
It follows from the equality that $f$ strongly attains its norm at $(p,r)$ and $(r,q)$.
\end{proof}

\begin{theorem}\label{theorem:SNA-Der-dense}
Let $Y$ be a Banach space. If $\SA(X,Y) \cup \D(X,Y)$ is dense in $\Lip(X,Y)$ for some Banach space $X$, then $Y$ has the extremal \textup{RNP}.
\end{theorem}

\begin{proof}
We first show that $\SA(\R,Y) \cup \D(\R,Y)$ is not dense in $\Lip(\R,Y)$ when $Y$ fails the extremal \textup{RNP}. If $Y$ fails the extremal \textup{RNP}, then we can consider the Lipschitz map $g \in \Lip(\R,Y)$ defined by
$$
g(t) = \lim_n \int_0^t f_n(\tau) \chi_C(\tau) \,d\tau
$$
on $[0,1]$, with its values extended constantly on the endpoints of $[0,1]$ outside, where $(f_n,B_n)$ is the martingale constructed with $\eps>0$ in Lemma \ref{lemma:RNP-martingale} and $C$ is a fat Cantor set with a positive measure. From the construction of $g$, we have
$$
\lim_{h \to 0^+} \left\|\frac{g(h)-g(0)}{h}\right\| = 1.
$$
More precisely, recall the function defined in the proof of Theorem \ref{theorem:RNP-Der}, namely
\begin{equation}\label{equation:define-g0}
g_0(t) := \lim_n \int_0^t f_n(\tau)\, d\tau \qquad \text{for } t \in [0,1]
\end{equation}
and extended constantly on $\R \setminus [0,1]$. Given any $\eta>0$ we can choose $\delta>0$ such that
$$
\int_0^h \chi_C(\tau) \, d\tau > h(1- \eta)
$$
whenever $0 < h<\delta$. This shows that
$$
\frac{\displaystyle \left\|g(h)- \lim_n \int_0^h f_n(\tau) \,d\tau \right\|}{h} = \frac{\displaystyle \left\| \lim_n \int_0^h f_n(\tau) [1-\chi_C(\tau)]\,d\tau \right\|}{h} < \eta,
$$
and the argument in the proof of Theorem \ref{theorem:RNP-Der} gives the conclusion. To see that $\|g\|=1$, fix any $t_1 \leq t_2$ in $[0,1]$ so that we obtain
\begin{align*}
\|g(t_2)-g(t_1)\| &= \lim_n \left\| \int_{t_1}^{t_2} f_n(\tau) \chi_C(\tau) \,d\tau \right\| \leq \int_{t_1}^{t_2} \chi_C(\tau) \,d\tau \leq t_2-t_1.
\end{align*}
Suppose now that $g \in \overline{\SA(\R,Y)}$ and there exists $f \in \SA(\R,Y)$ with $\|f\|=1$ such that $\|f-g\|<1/2$. Let $t_1 < t_2$ in $\R$ be such that $\|f(t_2)-f(t_1)\|= t_2-t_1$. By Lemma \ref{lemma:SNA-line}, we have $\|f(s_2)-f(s_1)\|=|s_2-s_1|$ for any $s_1,s_2 \in [t_1,t_2]$. Therefore,
$$
\|g(s_2)-g(s_1)\| \geq \|f(s_2)-f(s_1)\| - \|f-g\||s_2-s_1| >0
$$
for any $s_1 \neq s_2 \in [t_1,t_2]$. However, by the choice of $C$, this is impossible as every interval contains a subinterval which does not intersect $C$.

In order to prove $g \notin \overline{\D(\R,Y)}$, assume there exists $f \in \D(\R,Y)$ with $\|f\|=1$ such that $\|f-g\|<\eps/6$, say $f'(t)$ exists with $\|f'(t)\|=1$ for some $t \in [0,1]$. So there exists $\delta>0$ such that
$$
\left\| \frac{f(t+h)-f(t)}{h} - f'(t) \right\| < \dfrac{\eps}{6}
$$
whenever $0<|h|<\delta$. Observe here that we can choose $0<\delta_0<\delta$ such that
\begin{equation}\label{equation:Cantor-approx}
\left\| \dfrac{g(t+h)-g(t)}{h} - \dfrac{g_0(t+h)-g_0(t)}{h} \right\| < \dfrac{\eps}{6}
\end{equation}
for every $0<|h|<\delta_0$. Indeed, $\|f-g\|<\eps/6$ only makes sense when
$$
\left\| \dfrac{g(t+h)-g(t)}{h} \right\| > 1- \dfrac{\eps}{6}
$$
for small enough $h$, and thus there exists $\delta_0>0$ such that
$$
\int_{\min \{t,t+h\}}^{\max \{t,t+h\}} [ 1 - \chi_C(\tau) ] \, d\tau < \dfrac{h \eps}{6}
$$
whenever $0 < |h|<\delta_0$. On the other hand, there exists $0<|h_0|<\delta_0$ such that
$$
\left\| \dfrac{g_0(t+h_0)-g_0(t)}{h_0} - f'(t) \right\| \geq \dfrac{\eps}{2},
$$
since $g_0$ does not have any $\eps/2$-differentiability points in $[0,1]$ (see the proof of Theorem \ref{theorem:RNP-Der} or \cite[Theorem 5.21]{BL}). Thus we have
$$
\|g-f\| \geq \left\| \dfrac{g_0(t+h_0)-g_0(t)}{h_0} -f'(t) \right\| - \left\| \dfrac{f(t+h_0)-f(t)}{h_0} - f'(t) \right\| - \dfrac{\eps}{6} \geq \dfrac{\eps}{6},
$$
which leads to a contradiction.

Let us now consider $\tilde{g} \in \Lip(X,Y)$ defined for each $x \in X$ by
$$
\tilde{g}(x) := g(z_0^*(x)),
$$
where $z_0^* \in S_{X^*}$ attains its norm at $z \in S_X$ as in Theorem \ref{theorem:RNP-Der}. It is clear that $\|\tilde{g}\|=1$. We will show that $\tilde{g}$ cannot be approximated by Lipschitz maps in $\SA(X,Y) \cup \D(X,Y)$. Although the overall progression is similar, we provide a detailed proof due to significant differences. First, assume that there exists $\tilde{f} \in \SA(X,Y)$ with $\|\tilde{f}\|=1$ such that $\|\tilde{f}-\tilde{g}\|<1/2$. Let us say $(p,q) \in \widetilde{X}$ satisfy that $\|\tilde{f}(p)-\tilde{f}(q)\| = \|p-q\|$. By applying Lemma \ref{lemma:SNA-line}, for every distinct elements $r,s \in X$ which lie in the line joining $p$ and $q$, we have
\begin{align*}
\|\tilde{g}(r)-\tilde{g}(s)\| &\geq \|\tilde{f}(r)-\tilde{f}(s)\| - \|\tilde{f}-\tilde{g}\| \|r-s\| \\
&= \|r-s\| - \|\tilde{f}-\tilde{g}\|\|r-s\| > 0.
\end{align*}
But this is impossible from the construction of $\tilde{g}$. It remains to show that $\tilde{g} \notin \overline{\D(X,Y)}$. Assume again that there exists $\tilde{f} \in \D(X,Y)$ with $\|\tilde{f}\|=1$ such that $\|\tilde{f}-\tilde{g}\|<\eps/10$, and that $\tilde{f}'(x,e)$ exists with $\|\tilde{f}'(x,e)\|=1$ for some $x \in X$ and $e \in S_X$. We may find $\delta>0$ so that
$$
\left\| \frac{\tilde{f}(x+he)-\tilde{f}(x)}{h} - \tilde{f}'(x,e) \right\| < \dfrac{\eps}{10}
$$
whenever $0<|h|<\delta$. This gives that we may assume $z_0^*(e) > 1 - \eps/5$ since
$$
|z_0^*(e)| \|g\| \geq \left\| \dfrac{g(z_0^*(x+he))-g(z_0^*(x))}{h} \right\| \geq \left\| \frac{\tilde{f}(x+he)-\tilde{f}(x)}{h} \right\| -\|\tilde{f}-\tilde{g}\| > 1- \dfrac{\eps}{5},
$$
and by replacing $e$ with $-e$ if necessary. Let $\tilde{g_0}(x) := g_0(z_0^*(x))$ where $g_0$ is defined as in (\ref{equation:define-g0}). Then, we can choose $0<\delta_0<\delta$ such that
$$
\left\| \dfrac{g_0(z_0^*(x+he))-g_0(z_0^*(x))}{hz_0^*(e)} - \dfrac{g(z_0^*(x+he))-g(z_0^*(x))}{hz_0^*(e)} \right\| < \dfrac{\eps}{10}
$$
whenever $0<|h|<\delta_0$, following the argument in \eqref{equation:Cantor-approx}. Finally, applying again that $g_0$ has no $\eps/2$-differentiability points, we can find $0<|h_0|<\delta_0$ such that
$$
\left\| \dfrac{g_0(z_0^*(x+h_0e))-g_0(z_0^*(x))}{h_0z_0^*(e)} - \widetilde{f}'(x,e) \right\| \geq \dfrac{\eps}{2}.
$$
This leads to the inequality
\begin{align*}
\dfrac{\eps}{10} &\geq \|\tilde{g}-\tilde{f}\| \geq \left\| \dfrac{\tilde{g}(x+h_0e)-\tilde{g}(x)}{h_0} - \tilde{f}'(x,e) \right\| - \left\| \dfrac{\tilde{f}(x+h_0e)-\tilde{f}(x)}{h_0} - \tilde{f}'(x,e) \right\| \\
&\geq \left\| \dfrac{\tilde{g_0}(x+h_0e)-\tilde{g_0}(x)}{h_0z_0^*(e)} - \tilde{f}'(x,e) \right\| - \left\| \dfrac{\tilde{g_0}(x+h_0e)-\tilde{g_0}(x)}{h_0z_0^*(e)} - \dfrac{\tilde{g}(x+h_0e)-\tilde{g}(x)}{h_0z_0^*(e)} \right\| \\
&\quad- \left\| \dfrac{\tilde{g}(x+h_0e)-\tilde{g}(x)}{h_0z_0^*(e)} - \dfrac{\tilde{g}(x+h_0e)-\tilde{g}(x)}{h_0} \right\| - \left\| \dfrac{\tilde{f}(x+h_0e)-\tilde{f}(x)}{h_0} - f'(x,e) \right\| \\
&> \dfrac{\eps}{2} - \dfrac{\eps}{10} -\dfrac{\eps}{5} - \dfrac{\eps}{10} = \dfrac{\eps}{10},
\end{align*}
which is absurd.
\end{proof}

Again from the renorming perspective, we are able to obtain the following characterization on the \textup{RNP} when the domain space is $\R$.

\begin{cor}
Let $Y$ be a Banach space. Then, the following statements are equivalent.
\begin{enumerate}
\item[\textup{(a)}] $Y$ has the \textup{RNP}.
\item[\textup{(b)}] $\SA(\R,Z) \cup \D(\R,Z)$ is dense in $\Lip(\R,Z)$ for every renorming $Z$ of $Y$.
\end{enumerate}
\end{cor}

\begin{proof}
(a)$\Rightarrow$(b) is a consequence of Theorem \ref{theorem:CCM}.(b). In fact, we have $\SA(\R,Z) \subseteq \D(\R,Z)$ in this case. (b)$\Rightarrow$(a) follows from Proposition \ref{prop:renorm} and Theorem \ref{theorem:SNA-Der-dense}.
\end{proof}

For duals of separable spaces, we have the following characterization.

\begin{cor}
Let $Y$ be a separable Banach space. Then, the following statements are equivalent.
\begin{enumerate}
\item[\textup{(a)}] $Y^*$ has the \textup{RNP}.
\item[\textup{(b)}] $\SA(\R,Y^*) \cup \D(\R,Y^*)$ is dense in $\Lip(\R,Y^*)$.
\end{enumerate}
\end{cor}

Given that it remains unknown whether $\D(X,Y)$ is always dense in $\Lip(X,Y)$ whenever $Y$ has the \textup{RNP}, the general statement of a vector-valued version in terms of the \textup{RNP} is provided as below.

\begin{cor}
Let $X$ and $Y$ be Banach spaces. If $\SA(X,Z) \cup \D(X,Z)$ is dense in $\Lip(X,Z)$ for every renorming $Z$ of $Y$, then $Y$ has the \textup{RNP}.
\end{cor}

\ \

\section{On the maximal approximation by affine property}\label{section:maximal-AAP}

As previously noted, it was studied in \cite{BJLPS} the approximability of Lipschitz maps by affine functions, a property known as the \textup{AAP}. In our context, we aim to determine whether the same approximation can be obtained under a specific condition where the affine function has a slope sufficiently close to that of the approximating Lipschitz map, and this concept will be formalized by the following definition.

\begin{definition}
We say that a pair of Banach spaces $(X,Y)$ has the \emph{maximal approximation by affine property} \textup{(}\emph{maximal} $\AAP$ in short\textup{)} if for every $\eps>0$, a Lipschitz map $f \in \Lipn(X,Y)$ and a ball $B \subseteq X$, there exists another ball $B_1 \subseteq B$ and an affine function $g: B_1 \to Y$ such that
$$
\Lipn(g) > \Lipn(f|_B) - \eps \quad \text{and} \quad \sup_{x \in B_1} \|g(x)-f(x)\| \leq \eps \diam(B_1) \Lipn(f).
$$
If we can find a function $c(\eps)>0$ such that $B_1$ can be chosen $\diam(B_1) \geq c\diam(B)$, then we say $(X,Y)$ has the \emph{maximal uniform approximation by affine property} \textup{(}\emph{maximal uniform} $\AAP$ in short\textup{)}.
\end{definition}

It is important to  note that this property does not guarantee that every Lipschitz map can be approximated by strongly norm attaining ones in the (semi-)norm sense, as the approximation in the preceding definition is not a Lipschitz number perturbation. The next result shows that for Lipschitz functionals on $\R$, we obtain a positive answer.

\begin{theorem}
The pair $(\R,\R)$ has the maximal $\AAP$.
\end{theorem}

\begin{proof}
Let $0<\eps<1$, an interval $I \subseteq \R$ containing 0, $f \in \Lipn(I,\R)$ with $\Lipn(f)=1$ be given. Since $f$ is differentiable almost everywhere, we may think $\|f'\|_\infty = 1$. Let us write
$$
f(t) = f(0) + \int_{0}^t f'(s)\,ds \qquad \text{for } t \in I.
$$
Choose a measurable set $M \subseteq I$ with $\lambda(M)>0$ such that $f'(t)> 1 -\eps$ for every $t \in M$ without loss of generality. We may find a subinterval $I_1 \subseteq I$ so that
$$
\lambda(M \cap I_1) > (1-\eps) \lambda(I_1)> 0
$$
by applying the Lebesgue's density theorem. Indeed, there exists $t_0 \in I$ such that
$$
\lim_{\delta \to 0^+} \frac{\lambda\bigl(M \cap [t_0-\delta,t_0+\delta]\bigr)}{2\delta} =1.
$$
For such $I_1 = [a,b]$, define an affine function $g:I_1 \to \R$ by $g(t)=t+z$ where $z$ is chosen such that
$$
z = \frac{f(a)+f(b)-a-b}{2}.
$$
It is clear that $\Lipn(g) > \Lipn(f|_I) - \eps$ and by the choice of $z$ we have
\begin{align*}
\sup_{t \in I_1} |g(t)-f(t)| &= \frac{1}{2} \int_{I_1} [g'(t)-f'(t)]\, dt \\
&\leq \frac{1}{2} \left[ \diam(I_1) - (1 - 3\eps+\eps^2) \diam(I_1) \right] \leq \frac{3}{2} \eps \diam(I_1)
\end{align*}
from the fact that 
\begin{align*}
\int_{I_1} f'(s)\, ds &= \int_{I_1 \cap M} f'(s) \, ds + \int_{I_1 \setminus M} f'(s) \, ds \\
&\geq (1-\eps) \diam(I_1) (1 - \eps) - \eps \diam(I_1) = (1 - 3\eps + \eps^2) \diam(I_1).
\end{align*}
This shows that $(\R,\R)$ has the maximal $\AAP$.
\end{proof}

However, we should not anticipate the same conclusion when considering the uniform version of $\AAP$, which was a primary focus of \cite{BJLPS}, in the context of maximal affine functions. To verify this, we need a straightforward technical lemma regarding fat Cantor sets.

\begin{lemma}\label{lemma:fat-Cantor}
For fixed $0<c<1$ and $0<k \leq 1/4$, let $C$ be a fat Cantor set constructed by cutting out each part of the length $kc^n/4^{n-1}$ \textup{(}$2^{n-1}$ times in total\textup{)} of the interval $[0,1]$ in each $n^{\textup{th}}$ step. Then, we have
$$
\lambda(C \cap [a,b]) > \frac{1}{2}|b-a|
$$
for every interval $[a,b] \subseteq [0,1]$ with $|b-a| \geq c$.
\end{lemma}

\begin{proof}
By a straightforward calculation, we have
$$
\lambda(C) = 1 - \frac{kc}{1-c/2} > 1-2kc
$$
from $0<c<1$. So it follows that
\begin{align*}
\lambda(C \cap [a,b]) &= \lambda([a,b]) - \lambda([a,b] \setminus C) \geq \lambda([a,b]) - 1 + \lambda(C) > \dfrac{1}{2} |b-a|
\end{align*}
since $|b-a| \geq c$ and $0< k \leq 1/4$.
\end{proof}

We are now ready to prove the desired result on the maximal uniform $\AAP$ for the pair $(\R,\R)$.

\begin{theorem}
The pair $(\R,\R)$ fails the maximal uniform $\AAP$. That is, there exists $\eps>0$ such that for every $0<c<1$, there are an interval $I \subseteq \R$ and a nonzero $f \in \Lipn(I,\R)$ such that whenever $I_1 \subseteq I$ is an interval with $\diam(I_1) \geq c \diam(I)$ and $g : I_1 \to \R$ is an affine fucntion such that $\Lipn(g) > \Lipn(f|_I) - \eps$, we have
$$
\sup_{t \in I_1} |g(t)-f(t)| \geq \eps \diam(I_1)\Lipn(f).
$$
\end{theorem}

\begin{proof}
For any fixed $0<c<1$, we construct $f \in \Lipn([0,1],\R)$ with $\Lipn(f)=1$ so that whenever an interval $[a,b] \subseteq [0,1]$ with $|b-a| \geq c$ and an affine funcion $g:[a,b] \to \R$ with $\Lipn(g) > 7/8$ are given, we have
$$
\sup_{t \in [a,b]} |g(t) - f(t) | > \frac{1}{8}|b-a|.
$$
Let $C$ be a fat Cantor set, such that the length of each removed interval is $c^n/4^n$ during the $n^\text{th}$ step. It is immediate from the calculation in Lemma \ref{lemma:fat-Cantor} that $\lambda(C)>0$. Now, define the Lipschitz map $f \in \Lipn([0,1],\R)$ by
$$
f(t) := t - \int_0^t \chi_C(\tau) \, d\tau \qquad \text{for } t \in [0,1].
$$
Since $\Lipn(f|_{[0,1]})=1$ and $f$ is non-decreasing, an interval $[a,b] \subseteq [0,1]$ with $|b-a|\geq c$ and an affine function $g(t) := x_0t+y_0$ defined on $[a,b]$ must satisfy that $x_0 = \Lipn(g) >7/8$. Thus for every choice of $y_0 \in \R$, we have
\begin{align*}
\sup_{t \in [a,b]} |g(t)-f(t)| &= \sup_{t \in [a,b]} \left| (x_0-1)t + \left( y_0 + \int_0^t \chi_C(\tau) \, d\tau \right) \right| \\
&\geq \frac{1}{2} \left| (x_0-1)(b-a) + \int_a^b \chi_C(\tau) \, d\tau \right| \\
&\geq \frac{1}{2} \lambda(C \cap [a,b]) - \frac{1}{16}|b-a| \\
&> \frac{1}{8} |b-a|
\end{align*}
as claimed, where the first inequality is optimal when $y_0$ has the value
$$
y_0 = - \frac{1}{2} \left[ (x_0-1)(b-a) + \int_a^b \chi_C(\tau) \,d\tau \right],
$$
and the last inequality is from Lemma \ref{lemma:fat-Cantor}.
\end{proof}


\ \


\end{document}